\documentclass[12pt,a4paper,twoside]{article}
\usepackage[colorlinks=true,linkcolor=blue,citecolor=blue]{hyperref}
\usepackage{amsmath,mathrsfs,multicol,amsfonts,mathtools,amsthm, bm,fancyhdr,lastpage,xcolor,titlesec,cleveref}
\usepackage[english]{babel}
\usepackage{nicematrix}
\usepackage{amssymb}
\usepackage[bottom,multiple]{footmisc}
\usepackage[top=2.5cm,bottom=2.5cm,left=2.5cm,right=2.5cm]{geometry}
\definecolor{myeditcolor}{named}{blue}
\titleformat{\section}[block]{\bfseries\large}{\thesection. }{2pt}{}
\theoremstyle{definition}
\newtheorem{theorem}{Theorem}[section]
\newtheorem{definition}[theorem]{Definition}

\newtheorem{remark}[theorem]{Remark}

\begin{document}
\thispagestyle{empty}
\begin{center}
\noindent\textbf{{\Large Lorentzian Algebraic Ricci Solitons on the Heisenberg Group}}
\end{center}
\begin{center} \noindent Youssef Ayad\footnote{corresponding author: youssef.ayad@edu.umi.ac.ma}\\
\small{\textit{$^1$Faculty of sciences, Moulay Ismail University of Mekn\`{e}s}}\\
\small{\textit{B.P. 11201, Zitoune, Mekn\`{e}s}, Morocco}
\end{center}
\begin{abstract}
There exist three nonequivalent left invariant Lorentzian metrics on the Heisenberg group $H_{2n+1}$, or equivalently, three nonequivalent Lorentzian inner products on the Heisenberg Lie algebra $\mathfrak{h}_{2n+1}$, denoted by $\mu$, $\nu$, and $\phi$. We show that, in a specific case, $\mu$ is an algebraic Ricci soliton that is shrinking. Moreover, $\nu$ is an algebraic Ricci soliton only on the three-dimensional Heisenberg Lie algebra $\mathfrak{h}_3$ and it is shrinking. Finally, we show that $\phi$ is a steady algebraic Ricci soliton on $\mathfrak{h}_{2n + 1}$ for $n > 1$. However, for $n = 1$, $\phi$ is flat.
\end{abstract}
\begin{center}
\end{center}
\begin{center}
\end{center}
\section{Introduction}
\hspace*{3mm}    A Ricci soliton is a generalization of a Ricci-flat manifold and plays a central role in the study of the Ricci flow, an evolution equation for the Ricci curvature of a pseudo-Riemannian manifold. The Ricci soliton equation describes a manifold where the Ricci curvature evolves in a way that is partially balanced by the flow generated by a vector field. Specifically, a pseudo-Riemannian manifold $(M, g)$ is said to be a Ricci soliton if there exists a vector field $X$ and a constant $\eta$ such that the Ricci curvature $\rho$ satisfies the equation
$$\rho + \mathscr{L}_Xg = \eta g,$$
where $\mathscr{L}_Xg$ denotes the Lie derivative of the metric $g$ with respect to the vector field $X$. If $\eta = 0$, the soliton is said to be steady; if $\eta > 0$, it is shrinking; and if $\eta < 0$, it is expanding. Ricci solitons are self-similar solutions of Hamilton's Ricci flow, and they can also be viewed as its fixed points.

An algebraic Ricci soliton is a special type of Ricci soliton on a Lie group $G$ equipped with a left invariant pseudo-Riemannian metric $g$; moreover, $g$ is an algebraic Ricci soliton if its Ricci operator $\operatorname{Ric}$ satisfies the following equation
$$\operatorname{Ric} = \eta\operatorname{Id}_{\mathfrak{g}} + D,$$
where $\eta \in \mathbb{R}$ and $D$ is a derivation of the Lie algebra $\mathfrak{g}$ of $G$. This notion was introduced by Lauret in \cite{lauret2001ricci}, where he established a fundamental connection between solvsolitons and Ricci solitons on homogeneous Riemannian manifolds. More precisely, he proved that any Riemannian solvsoliton metric is a Ricci soliton. In a later work \cite{lauret2011ricci}, Lauret extended his study of Ricci solitons on solvable Lie groups by investigating the structure of solsolitons. He showed that every solsoliton arises, up to isometry, from a nilsoliton together with an abelian Lie algebra of symmetric derivations of the associated metric Lie algebra. He also proved the uniqueness of solsolitons on a fixed solvable Lie group up to isometry and scaling, and classified all solsolitons of dimension at most four. In \cite{jablonski2011concerning,jablonski2014homogeneous,jablonski2015homogeneous}, Jablonski significantly extended Lauret's work by proving that homogeneous Ricci solitons are algebraic, characterizing solvable Lie groups admitting Ricci soliton metrics, and establishing several structural and rigidity results for homogeneous Ricci solitons.

The notion of an algebraic Ricci soliton was extended to the pseudo-Riemannian setting in \cite{batat2017algebraic}, where Batat and Onda investigated algebraic Ricci solitons on three-dimensional Lorentzian Lie groups. They obtained a complete classification of such solitons and showed that, in contrast to the Riemannian case, Lorentzian Ricci solitons are not necessarily algebraic. The classification of Lorentzian algebraic Ricci solitons on four-dimensional nilpotent Lie groups was carried out in \cite{aitbenhaddou2026lorentzian}, building on the classification of left invariant Lorentzian metrics on four-dimensional nilpotent Lie groups established in \cite{bokan2015lorentz}. More generally, a complete study of four-dimensional Lorentzian algebraic Ricci solitons was given in \cite{garcia}. As a consequence, the authors proved that, in sharp contrast to the Riemannian setting, every connected and simply connected four-dimensional Lie group admits a left invariant Lorentzian metric that is a Ricci soliton.

Also, the classification of Lorentzian algebraic Ricci solitons on Lie groups with respect to the canonical and Kobayashi-Nomizu connections was obtained in dimension three in \cite{wang2022canonical}, and in dimension four and in the nilpotent case in \cite{ayad2026canonical}.

The classification of left invariant Lorentzian metrics on the Heisenberg group was obtained in \cite{vukmirovic2015classification}. This problem is equivalent to classifying Lorentzian inner products on the Heisenberg Lie algebra. Vukmirovi{\'c} showed in \cite{vukmirovic2015classification} that there exist three nonequivalent Lorentzian inner products on $\mathfrak{h}_{2n+1}$, which we denote by $\mu$, $\nu$, and $\phi$. However, the Levi-Civita connections and Ricci operators associated with these metrics were not presented in detail. Since the procedure in the Lorentzian setting differs from that in the Riemannian case, the Levi-Civita connection and Ricci operator must be carefully addressed in the Lorentzian setting. We then give the correct results, which also differ from those presented in \cite{nasehi2019geometry}.

In this paper, we describe the Levi-Civita connection and the Ricci operator of $(\mathfrak{h}_{2n + 1}, \mu)$, $(\mathfrak{h}_{2n + 1}, \nu)$ and $(\mathfrak{h}_{2n + 1}, \phi)$. This leads us to find the following main result: 

We show that, in a specific case, $\mu$ is an algebraic Ricci soliton that is shrinking. Moreover, $\nu$ is an algebraic Ricci soliton only on the three-dimensional Heisenberg Lie algebra $\mathfrak{h}_3$ and it is shrinking. Finally, we prove that $\phi$ is a steady algebraic Ricci soliton on $\mathfrak{h}_{2n + 1}$ for $n > 1$. However, for $n = 1$, $\phi$ is flat.
\section{Preliminaries}
Let $(G, g)$ be a Lorentzian Lie group, where the Lorentzian metric $g$ is left invariant and let $\left(\mathfrak{g}, \langle ., .\rangle\right)$ be its associated Lorentzian Lie algebra of dimension $m$. Consider $\nabla$ to be the Levi-Civita connection associated to $\left(\mathfrak{g}, \langle \cdot, \cdot\rangle\right)$; It is a standard fact that $\nabla$ is characterized by the following Koszul formula
$$2\langle \nabla_u v, w\rangle = \langle [u, v], w\rangle + \langle [w, u], v\rangle + \langle [w, v], u\rangle, \quad \forall u, v, w \in \mathfrak{g}.$$
The Riemann curvature tensor $R$ of $\left(\mathfrak{g}, \langle \cdot, \cdot\rangle\right)$ assigns to each pair $u, v \in \mathfrak{g}$ the linear transformation
$$R_{uv} = \nabla_{[u, v]} - [\nabla_u, \nabla_v].$$
More precisely, we have
$$R_{uv}w = \nabla_{[u, v]}w - \nabla_u\nabla_v w + \nabla_v\nabla_u w, \quad \forall w \in \mathfrak{g}.$$
Let $\mathscr{B} = \left\lbrace v_1, \ldots, v_m\right\rbrace$ be an orthonormal basis of $\left(\mathfrak{g}, \langle \cdot, \cdot\rangle\right)$, that is, it satisfies 
$$\langle v_i, v_j \rangle = 0 \quad \forall i \neq j, \quad \langle v_i, v_i \rangle = 1 \quad \forall i = 1, ..., m - 1, \quad \langle v_m, v_m \rangle = -1.$$
The structure constants $\xi_{ijk}$ of $\left(\mathfrak{g}, \langle ., .\rangle\right)$ are defined by
$$\xi_{ijk} = \langle [v_i, v_j], v_k\rangle.$$
By the Koszul formula, the Levi-Civita connection of $\left(\mathfrak{g}, \langle \cdot, \cdot\rangle\right)$ is described by
\begin{equation}\label{Levi}
\nabla_{v_i}v_j = \displaystyle{\sum_{k = 1}^{m} \langle v_k, v_k\rangle \frac{1}{2}\left(\xi_{ijk} - \xi_{jki} + \xi_{kij}\right)v_k}.
\end{equation}
The Ricci tensor $\rho$ of $\left(\mathfrak{g}, \langle \cdot, \cdot\rangle\right)$ is defined by
$$\rho(u, v) = \displaystyle{\sum_{k = 1}^{m} \langle v_k, v_k\rangle}\langle R_{v_ku}v_k, v\rangle.$$
The Ricci operator $\operatorname{Ric}$ of $\left(\mathfrak{g}, \langle \cdot, \cdot\rangle\right)$ is defined by
$$\rho(u, v) = \langle \operatorname{Ric}(u), v\rangle,$$
or equivalently
$$\operatorname{Ric}(u) = \displaystyle{\sum_{k = 1}^{m} \langle v_k, v_k\rangle R_{v_ku}v_k}.$$
\begin{definition}
\begin{enumerate}
\item A Lorentzian inner product on the Lie algebra $\mathfrak{g}$ is called flat if its Riemann curvature tensor is identically zero.
\item A Lorentzian inner product on the Lie algebra $\mathfrak{g}$ is called an algebraic Ricci soliton if its Ricci operator $\operatorname{Ric}$ satisfies the following equality
$$\operatorname{Ric} = \eta\operatorname{Id}_{\mathfrak{g}} + D,$$
where $\eta$ is a real number and $D$ is a derivation of $\mathfrak{g}$, that is
$$[D(u), v] + [u, D(v)] = D[u, v], \quad \forall u, v \in \mathfrak{g}.$$
\end{enumerate}
\end{definition}
\section{The Heisenberg Lie algebra}
The Heisenberg Lie group $H_{2n + 1}$ and its associated Heisenberg Lie algebra $\mathfrak{h}_{2n + 1}$ are introduced in \cite{vukmirovic2015classification}. Moreover $\mathfrak{h}_{2n + 1}$ has a basis $B = \left\lbrace f_1, g_1, \ldots, f_n, g_n, z\right\rbrace$ with nonzero commutators
$$[f_i, g_i] = z, \quad i = 1, \ldots, n.$$
\begin{theorem}
\cite{vukmirovic2015classification} Any Lorentzian inner product on $\mathfrak{h}_{2n + 1}$, up to Lie algebra automorphism, is represented in the basis $B$ ($B$ is a special basis satisfying $(1)$ in \cite{vukmirovic2015classification}, as mentioned by the author in \cite{vukmirovic2015classification}) by one of the following matrices
$$\begin{aligned}
\mu &= \begin{bmatrix}
D_{n - 1}(\sigma) & 0\\
0 & S
\end{bmatrix}, \; S = \operatorname{diag}\left\lbrace 1, 1, -\lambda\right\rbrace, \; \lambda > 0,\\
\nu &= \begin{bmatrix}
D_{n - 1}(\sigma) & 0\\
0 & S
\end{bmatrix}, \; S = \operatorname{diag}\left\lbrace 1, -1, \lambda\right\rbrace, \; \lambda > 0,\\
\phi &= \begin{bmatrix}
D_{n - 1}(\sigma) & 0\\
0 & S
\end{bmatrix}, \; S = \begin{bmatrix}
1 & 0 & 0\\
0 & 0 & 1\\
0 & 1 & 0
\end{bmatrix},
\end{aligned}$$
where $D_{n - 1}(\sigma) = \operatorname{diag}\left\lbrace \sigma_1, \sigma_1, \ldots, \sigma_{n - 1}, \sigma_{n - 1}\right\rbrace$ and $\sigma_1 \geq \ldots \geq \sigma_{n - 1} \geq 1$.
\end{theorem}
\subsection{The Lorentzian Lie algebra $\left(\mathfrak{h}_{2n + 1}, \mu\right)$}
An orthonormal basis $\mathscr{B} = \left\lbrace v_1, \ldots, v_{2n + 1}\right\rbrace$ of $\left(\mathfrak{h}_{2n + 1}, \mu\right)$ is given by
$$\begin{aligned}
v_1 &= \frac{1}{\sqrt{\sigma_1}}f_1,
&\qquad
v_2 &= \frac{1}{\sqrt{\sigma_1}}g_1, \\[2pt]
v_3 &= \frac{1}{\sqrt{\sigma_2}}f_2,
&\qquad
v_4 &= \frac{1}{\sqrt{\sigma_2}}g_2, \\
&\vdots \\
v_{2n-3} &= \frac{1}{\sqrt{\sigma_{n-1}}}f_{n-1},
&\qquad
v_{2n-2} &= \frac{1}{\sqrt{\sigma_{n-1}}}g_{n-1}, \\[2pt]
v_{2n-1} &= f_n,
&\qquad
v_{2n} &= g_n, \\[2pt]
v_{2n+1} &= \frac{1}{\sqrt{\lambda}}z.
\end{aligned}$$
The bracket in the basis $\mathscr{B}$ is given by
$$\begin{aligned}
[v_1, v_2] &= \frac{\sqrt{\lambda}}{\sigma_1}v_{2n + 1},\\
[v_3, v_4] &= \frac{\sqrt{\lambda}}{\sigma_2}v_{2n + 1},\\
&\vdots \\
[v_{2n - 3}, v_{2n - 2}] &= \frac{\sqrt{\lambda}}{\sigma_{n - 1}}v_{2n + 1},\\
[v_{2n - 1}, v_{2n}] &= \sqrt{\lambda}v_{2n + 1}.
\end{aligned}$$
The structure constants are
$$\begin{aligned}
\xi_{12(2n + 1)} &= -\xi_{21(2n + 1)} = \frac{-\sqrt{\lambda}}{\sigma_1},\\
\xi_{34(2n + 1)} &= -\xi_{43(2n + 1)} = \frac{-\sqrt{\lambda}}{\sigma_2},\\
&\vdots \\
\xi_{(2n - 3)(2n - 2)(2n + 1)} &= -\xi_{(2n - 2)(2n - 3)(2n + 1)} = \frac{-\sqrt{\lambda}}{\sigma_{n - 1}},\\
\xi_{(2n - 1)(2n)(2n + 1)} &= -\xi_{(2n)(2n - 1)(2n + 1)} = -\sqrt{\lambda}.
\end{aligned}$$
\begin{theorem}
The nonzero terms in the Levi-Civita connection of $\left(\mathfrak{h}_{2n + 1}, \mu\right)$ are
$$
\begin{aligned}
\nabla_{v_i}v_i &= 0, \qquad i=1,\ldots,2n+1,\\[2mm]
\nabla_{v_1}v_2
&=-\nabla_{v_2}v_1
=\frac{\sqrt{\lambda}}{2\sigma_1}v_{2n+1},\\
\nabla_{v_1}v_{2n+1}
&=\nabla_{v_{2n+1}}v_1
=\frac{\sqrt{\lambda}}{2\sigma_1}v_2,\\
\nabla_{v_2}v_{2n+1}
&=\nabla_{v_{2n+1}}v_2
=-\frac{\sqrt{\lambda}}{2\sigma_1}v_1,\\[2mm]
\nabla_{v_3}v_4
&=-\nabla_{v_4}v_3
=\frac{\sqrt{\lambda}}{2\sigma_2}v_{2n+1},\\
\nabla_{v_3}v_{2n+1}
&=\nabla_{v_{2n+1}}v_3
=\frac{\sqrt{\lambda}}{2\sigma_2}v_4,\\
\nabla_{v_4}v_{2n+1}
&=\nabla_{v_{2n+1}}v_4
=-\frac{\sqrt{\lambda}}{2\sigma_2}v_3,\\
&\hspace{2cm}\vdots\\[-1mm]
\nabla_{v_{2n-3}}v_{2n-2}
&=-\nabla_{v_{2n-2}}v_{2n-3}
=\frac{\sqrt{\lambda}}{2\sigma_{n-1}}v_{2n+1},\\
\nabla_{v_{2n-3}}v_{2n+1}
&=\nabla_{v_{2n+1}}v_{2n-3}
=\frac{\sqrt{\lambda}}{2\sigma_{n-1}}v_{2n-2},\\
\nabla_{v_{2n-2}}v_{2n+1}
&=\nabla_{v_{2n+1}}v_{2n-2}
=-\frac{\sqrt{\lambda}}{2\sigma_{n-1}}v_{2n-3}, \\[2mm]
\nabla_{v_{2n-1}}v_{2n}
&=-\nabla_{v_{2n}}v_{2n-1}
=\frac{\sqrt{\lambda}}{2}v_{2n+1},\\
\nabla_{v_{2n-1}}v_{2n+1}
&=\nabla_{v_{2n+1}}v_{2n-1}
=\frac{\sqrt{\lambda}}{2}v_{2n},\\
\nabla_{v_{2n}}v_{2n+1}
&=\nabla_{v_{2n+1}}v_{2n}
=-\frac{\sqrt{\lambda}}{2}v_{2n-1}.
\end{aligned}
$$
\end{theorem}
\begin{proof}
Since the structure constants are well described, the result follows by using the formula (\ref{Levi}).
\end{proof}
\subsubsection{The Ricci operator of $\left(\mathfrak{h}_{2n + 1}, \mu\right)$}
\begin{theorem}\label{ric}
The Ricci operator of $\left(\mathfrak{h}_{2n + 1}, \mu\right)$ is represented in the basis $\mathscr{B}$ by the following diagonal matrix
$$\operatorname{Ric} = \operatorname{diag}\left\lbrace \frac{\lambda}{2\sigma_1^2}, \frac{\lambda}{2\sigma_1^2}, \frac{\lambda}{2\sigma_2^2}, \frac{\lambda}{2\sigma_2^2}, \ldots, \frac{\lambda}{2\sigma_{n - 1}^2}, \frac{\lambda}{2\sigma_{n - 1}^2}, \frac{\lambda}{2}, \frac{\lambda}{2}, \delta\right\rbrace,$$
where $\delta$ is given by $\delta = \left(\displaystyle{\sum_{k = 1}^{n - 1}\frac{-\lambda}{2\sigma_k^2}}\right) - \frac{\lambda}{2}$.
\end{theorem}
\begin{proof}
We recall that the Ricci operator is defined by
$$\operatorname{Ric}(u) = \displaystyle{\sum_{k = 1}^{2n + 1} \mu(v_k, v_k) R_{v_ku}v_k},$$ 
where 
$$R_{v_ku}v_k = \nabla_{[v_k, u]}v_k - \nabla_{v_k}\nabla_u v_k + \nabla_u\nabla_{v_k} v_k.$$
Since $\nabla_{v_k} v_k = 0, \forall k = 1, \ldots, 2n + 1$, then the curvature simplifies to
$$R_{v_ku}v_k = \nabla_{[v_k, u]}v_k - \nabla_{v_k}\nabla_u v_k.$$
Now, let us compute $\operatorname{Ric}(v_1)$, we have
$$\operatorname{Ric}(v_1) = \displaystyle{\sum_{k = 1}^{2n + 1} \mu(v_k, v_k) R_{v_kv_1}v_k}.$$
In view of the bracket in $\mathscr{B}$ and the Levi-Civita connection, the only nonzero terms in $R_{v_kv_1}v_k$ are the following
$$\begin{aligned}
R_{v_2v_1}v_2 &= \nabla_{[v_2, v_1]}v_2 - \nabla_{v_2}\nabla_{v_1}v_2\\
&= \frac{-\sqrt{\lambda}}{\sigma_1}\nabla_{v_{2n + 1}}v_2 - \frac{\sqrt{\lambda}}{2\sigma_1}\nabla_{v_2}v_{2n + 1}\\
&= \frac{-3\sqrt{\lambda}}{2\sigma_1} \times \frac{-\sqrt{\lambda}}{2\sigma_1}v_1 = \frac{3\lambda}{4\sigma_1^2}v_1.\\
R_{v_{2n + 1}v_1}v_{2n + 1} &= - \nabla_{v_{2n + 1}}\nabla_{v_1}v_{2n + 1}\\
&= \frac{-\sqrt{\lambda}}{2\sigma_1}\nabla_{v_{2n + 1}}v_2 = \frac{\lambda}{4\sigma_1^2}v_1.
\end{aligned}$$
Since $v_{2n + 1}$ is timelike, we have that
$$\operatorname{Ric}(v_1) = R_{v_2v_1}v_2 - R_{v_{2n + 1}v_1}v_{2n + 1} = \frac{\lambda}{2\sigma_1^2}v_1.$$
The same reasoning applies to the other terms, despite the timelike vector $v_{2n+1}$; for the latter, we have
\begin{eqnarray*}
\operatorname{Ric}(v_{2n + 1}) &=& \displaystyle{\sum_{k = 1}^{2n + 1} \mu(v_k, v_k) \left( -\nabla_{v_k}\nabla_{v_{2n + 1}}v_k\right)}\\
&=& -\displaystyle{\sum_{k = 1}^{2n}\nabla_{v_k}\nabla_{v_{2n + 1}}v_k}.
\end{eqnarray*}
One can see that 
$$\begin{aligned}
\nabla_{v_1}\nabla_{v_{2n + 1}}v_1 &= \frac{\sqrt{\lambda}}{2\sigma_1}\nabla_{v_1}v_2 = \frac{\lambda}{4\sigma_1^2}v_{2n + 1},\\
\nabla_{v_2}\nabla_{v_{2n + 1}}v_2 &= \frac{-\sqrt{\lambda}}{2\sigma_1}\nabla_{v_2}v_1 = \frac{\lambda}{4\sigma_1^2}v_{2n + 1},\\
&\vdots \\
\nabla_{v_{2n - 1}}\nabla_{v_{2n + 1}}v_{2n - 1} &= \frac{\sqrt{\lambda}}{2}\nabla_{v_{2n - 1}}v_{2n} = \frac{\lambda}{4}v_{2n + 1},\\
\nabla_{v_{2n}}\nabla_{v_{2n + 1}}v_{2n} &= \frac{-\sqrt{\lambda}}{2}\nabla_{v_{2n}}v_{2n - 1} = \frac{\lambda}{4}v_{2n + 1}.
\end{aligned}$$
Therefore
\begin{eqnarray*}
\operatorname{Ric}(v_{2n + 1}) &=& -\left(\frac{\lambda}{2\sigma_1^2} + \frac{\lambda}{2\sigma_2^2} + \ldots + \frac{\lambda}{2\sigma_{n - 1}^2} + \frac{\lambda}{2}\right)v_{2n + 1}\\
&=& \left(\left(\displaystyle{\sum_{k = 1}^{n - 1}\frac{-\lambda}{2\sigma_k^2}}\right) - \frac{\lambda}{2}\right)v_{2n + 1} = \delta v_{2n + 1}.
\end{eqnarray*}
\end{proof}
\begin{remark}
If $n = 1$, we recover the Ricci operator of the Lorentzian inner product $\mu$ on the three-dimensional Heisenberg Lie algebra described (up to permutation) in \cite{ayad2026lorentzian} by
$$\operatorname{Ric} = \operatorname{diag}\left\lbrace \frac{\lambda}{2}, \frac{\lambda}{2}, \frac{-\lambda}{2}\right\rbrace.$$
\end{remark}
\begin{theorem}
The Lorentzian inner product $\mu$ is an algebraic Ricci soliton if and only if
$$\sigma_1 = \sigma_2 = \ldots = \sigma_{n - 1} = 1.$$
In this case we have 
$$\operatorname{Ric} = \operatorname{diag}\left\lbrace \frac{\lambda}{2}, \frac{\lambda}{2},\ldots, \frac{\lambda}{2}, \frac{-n\lambda}{2}\right\rbrace.$$
The algebraic Ricci soliton equation is given by
$$\operatorname{Ric} = \frac{(n + 2)\lambda}{2}I_{2n + 1} + D,$$
where
$$D = \operatorname{diag}\left\lbrace \frac{-(n + 1)\lambda}{2}, \frac{-(n + 1)\lambda}{2},\ldots, \frac{-(n + 1)\lambda}{2}, -(n + 1)\lambda\right\rbrace,$$
is a derivation of $\mathfrak{h}_{2n + 1}$ with respect to the orthonormal basis $\mathscr{B}$.
\end{theorem}
\begin{proof}
Assume that $\mu$ is an algebraic Ricci soliton, then there exists $\eta \in \mathbb{R}$ and a derivation $D$ of $\mathfrak{h}_{2n + 1}$ such that
$$\operatorname{Ric} = \eta I_{2n + 1} + D.$$
From this, it follows that
\begin{eqnarray*}
D &=& \operatorname{Ric} - \eta I_{2n + 1}\\
&=& \operatorname{diag}\left\lbrace \frac{\lambda}{2\sigma_1^2} - \eta, \frac{\lambda}{2\sigma_1^2} - \eta,\ldots, \frac{\lambda}{2\sigma_{n - 1}^2} - \eta, \frac{\lambda}{2\sigma_{n - 1}^2} - \eta, \frac{\lambda}{2} - \eta, \frac{\lambda}{2} - \eta, \delta - \eta\right\rbrace.
\end{eqnarray*}
Since $D$ is a derivation of $\mathfrak{h}_{2n + 1}$ that is diagonal, it obviously satisfies the derivation condition for all vanishing brackets. Next, $D$ satisfies the condition
\begin{eqnarray*}
& [D(v_1), v_2] + [v_1, D(v_2)] = D[v_1, v_2] \\ \Leftrightarrow & \left[ \left(\frac{\lambda}{2\sigma_1^2} - \eta\right)v_1, v_2\right] + \left[ v_1, \left(\frac{\lambda}{2\sigma_1^2} - \eta\right)v_2\right] = \frac{\sqrt{\lambda}}{\sigma_1}D(v_{2n + 1})\\
\Leftrightarrow& \left(\frac{\lambda}{\sigma_1^2} - 2\eta\right)[v_1, v_2] = \frac{\sqrt{\lambda}}{\sigma_1}(\delta - \eta)v_{2n + 1} \\
\Leftrightarrow& \left(\frac{\lambda}{\sigma_1^2} - 2\eta\right)\frac{\sqrt{\lambda}}{\sigma_1}v_{2n + 1} = \frac{\sqrt{\lambda}}{\sigma_1}(\delta - \eta)v_{2n + 1} \\
\Leftrightarrow& \frac{\lambda}{\sigma_1^2} - 2\eta = \delta - \eta \Leftrightarrow \eta = \frac{\lambda}{\sigma_1^2} - \delta.
\end{eqnarray*}
Applying the same reasoning to the other brackets, we obtain that
$$\begin{aligned}
\eta &= \frac{\lambda}{\sigma_i^2} - \delta, \quad i = 1, \ldots, n - 1\\
\eta &= \lambda - \delta.
\end{aligned}$$
If $\sigma_i \neq \sigma_j$, we have $\eta - \eta = 0 = \frac{\lambda}{\sigma_i^2} - \frac{\lambda}{\sigma_j^2}$, which is a contradiction. Since $\eta = \lambda - \delta$, then $\sigma_1 = \ldots = \sigma_{n - 1} = 1$. In this case, we have $\delta = \frac{-n\lambda}{2}$ and $\eta = \lambda - \delta = \frac{(n + 2)\lambda}{2}$. The Ricci operator become
$$\operatorname{Ric} = \operatorname{diag}\left\lbrace \frac{\lambda}{2}, \frac{\lambda}{2},\ldots, \frac{\lambda}{2}, \frac{-n\lambda}{2}\right\rbrace.$$
The algebraic Ricci soliton equation is given by
$$\operatorname{Ric} = \frac{(n + 2)\lambda}{2}I_{2n + 1} + D,$$
where
$$D = \operatorname{diag}\left\lbrace \frac{-(n + 1)\lambda}{2}, \frac{-(n + 1)\lambda}{2},\ldots, \frac{-(n + 1)\lambda}{2}, -(n + 1)\lambda\right\rbrace,$$
is a derivation of $\mathfrak{h}_{2n + 1}$ with respect to the orthonormal basis $\mathscr{B}$.
\end{proof}
\begin{remark}
For $n = 1$, we recover the algebraic Ricci soliton equation for the Lorentzian inner product $\mu$ on the three-dimensional Heisenberg Lie algebra given in \cite{ayad2026lorentzian} by
$$\operatorname{Ric} = \operatorname{diag}\left\lbrace \frac{\lambda}{2}, \frac{\lambda}{2}, \frac{-\lambda}{2}\right\rbrace = \frac{3\lambda}{2}I_3 + \operatorname{diag}\left\lbrace -\lambda, -\lambda, -2\lambda\right\rbrace.$$
\end{remark}
\subsection{The Lorentzian Lie algebra $\left(\mathfrak{h}_{2n + 1}, \nu\right)$}
An orthonormal basis $\mathscr{B} = \left\lbrace v_1, \ldots, v_{2n + 1}\right\rbrace$ of $\left(\mathfrak{h}_{2n + 1}, \nu\right)$ is given by
$$\begin{aligned}
v_1 &= \frac{1}{\sqrt{\sigma_1}}f_1,
&\qquad
v_2 &= \frac{1}{\sqrt{\sigma_1}}g_1, \\[2pt]
v_3 &= \frac{1}{\sqrt{\sigma_2}}f_2,
&\qquad
v_4 &= \frac{1}{\sqrt{\sigma_2}}g_2, \\
&\vdots \\
v_{2n-3} &= \frac{1}{\sqrt{\sigma_{n-1}}}f_{n-1},
&\qquad
v_{2n-2} &= \frac{1}{\sqrt{\sigma_{n-1}}}g_{n-1}, \\[2pt]
v_{2n-1} &= f_n,
&\qquad
v_{2n} &= \frac{1}{\sqrt{\lambda}}z, \\[2pt]
v_{2n+1} &= g_n.
\end{aligned}$$
The bracket in the basis $\mathscr{B}$ is given by
$$
[v_1, v_2] = \frac{\sqrt{\lambda}}{\sigma_1}v_{2n}, \qquad [v_3, v_4] = \frac{\sqrt{\lambda}}{\sigma_2}v_{2n}, \quad \ldots, \quad [v_{2n - 3}, v_{2n - 2}] = \frac{\sqrt{\lambda}}{\sigma_{n - 1}}v_{2n},$$
$$[v_{2n - 1}, v_{2n + 1}] = \sqrt{\lambda}v_{2n}.$$
The structure constants are
$$\begin{aligned}
\xi_{12(2n)} &= -\xi_{21(2n)} = \frac{\sqrt{\lambda}}{\sigma_1},\\
\xi_{34(2n)} &= -\xi_{43(2n)} = \frac{\sqrt{\lambda}}{\sigma_2},\\
&\vdots \\
\xi_{(2n - 3)(2n - 2)(2n)} &= -\xi_{(2n - 2)(2n - 3)(2n)} = \frac{\sqrt{\lambda}}{\sigma_{n - 1}},\\
\xi_{(2n - 1)(2n + 1)(2n)} &= -\xi_{(2n + 1)(2n - 1)(2n)} = \sqrt{\lambda}.
\end{aligned}$$
\begin{theorem}
The nonzero terms in the Levi-Civita connection of $\left(\mathfrak{h}_{2n + 1}, \nu\right)$ are
$$
\begin{aligned}
\nabla_{v_i}v_i &= 0, \qquad i=1,\ldots,2n+1,\\[2mm]
\nabla_{v_1}v_2
&=-\nabla_{v_2}v_1=\frac{\sqrt{\lambda}}{2\sigma_1}v_{2n},\\
\nabla_{v_1}v_{2n}
&=\nabla_{v_{2n}}v_1=\frac{-\sqrt{\lambda}}{2\sigma_1}v_2,\\
\nabla_{v_2}v_{2n}
&=\nabla_{v_{2n}}v_2=\frac{\sqrt{\lambda}}{2\sigma_1}v_1,\\[2mm]
\nabla_{v_3}v_4
&=-\nabla_{v_4}v_3=\frac{\sqrt{\lambda}}{2\sigma_2}v_{2n},\\
\nabla_{v_3}v_{2n}
&=\nabla_{v_{2n}}v_3=\frac{-\sqrt{\lambda}}{2\sigma_2}v_4,\\
\nabla_{v_4}v_{2n}
&=\nabla_{v_{2n}}v_4=\frac{\sqrt{\lambda}}{2\sigma_2}v_3,\\
&\hspace{2cm}\vdots\\[-1mm]
\nabla_{v_{2n-3}}v_{2n-2}
&=-\nabla_{v_{2n-2}}v_{2n-3}=\frac{\sqrt{\lambda}}{2\sigma_{n-1}}v_{2n},\\
\nabla_{v_{2n-3}}v_{2n}
&=\nabla_{v_{2n}}v_{2n-3}=\frac{-\sqrt{\lambda}}{2\sigma_{n-1}}v_{2n-2},\\
\nabla_{v_{2n-2}}v_{2n}
&=\nabla_{v_{2n}}v_{2n-2}=\frac{\sqrt{\lambda}}{2\sigma_{n-1}}v_{2n-3}, \\[2mm]
\nabla_{v_{2n-1}}v_{2n}
&=\nabla_{v_{2n}}v_{2n-1}=\frac{\sqrt{\lambda}}{2}v_{2n+1},\\
\nabla_{v_{2n-1}}v_{2n+1}
&=-\nabla_{v_{2n+1}}v_{2n-1}=\frac{\sqrt{\lambda}}{2}v_{2n},\\
\nabla_{v_{2n}}v_{2n+1}
&=\nabla_{v_{2n+1}}v_{2n}=\frac{\sqrt{\lambda}}{2}v_{2n-1}.
\end{aligned}
$$
\end{theorem}
\begin{proof}
Since the structure constants are well described, the result follows by using the formula (\ref{Levi}).
\end{proof}
\subsubsection{The Ricci operator of $\left(\mathfrak{h}_{2n + 1}, \nu\right)$}
\begin{theorem}
The Ricci operator of $\left(\mathfrak{h}_{2n + 1}, \nu\right)$ is represented in the basis $\mathscr{B}$ by the following diagonal matrix
$$\operatorname{Ric} = \operatorname{diag}\left\lbrace \frac{-\lambda}{2\sigma_1^2}, \frac{-\lambda}{2\sigma_1^2}, \frac{-\lambda}{2\sigma_2^2}, \frac{-\lambda}{2\sigma_2^2}, \ldots, \frac{-\lambda}{2\sigma_{n - 1}^2}, \frac{-\lambda}{2\sigma_{n - 1}^2}, \frac{\lambda}{2}, \theta, \frac{\lambda}{2}\right\rbrace,$$
where $\theta$ is given by $\theta = \left(\displaystyle{\sum_{k = 1}^{n - 1}\frac{\lambda}{2\sigma_k^2}}\right) - \frac{\lambda}{2}$.
\end{theorem}
\begin{proof}
The proof is similar to that of theorem \ref{ric}, the only Ricci curvature that needs an explicit clarification is the following
\begin{eqnarray*}
\operatorname{Ric}(v_{2n}) &=& \displaystyle{\sum_{k = 1}^{2n + 1} \nu(v_k, v_k) \left( -\nabla_{v_k}\nabla_{v_{2n}}v_k\right)}\\
&=& -\displaystyle{\sum_{k = 1}^{2n - 2}\nabla_{v_k}\nabla_{v_{2n}}v_k} - \nabla_{v_{2n - 1}}\nabla_{v_{2n}}v_{2n - 1} + \nabla_{v_{2n + 1}}\nabla_{v_{2n}}v_{2n + 1}.
\end{eqnarray*}
One can see that 
$$\begin{aligned}
\nabla_{v_1}\nabla_{v_{2n}}v_1 &= \frac{-\sqrt{\lambda}}{2\sigma_1}\nabla_{v_1}v_2 = \frac{-\lambda}{4\sigma_1^2}v_{2n},\\
\nabla_{v_2}\nabla_{v_{2n}}v_2 &= \frac{\sqrt{\lambda}}{2\sigma_1}\nabla_{v_2}v_1 = \frac{-\lambda}{4\sigma_1^2}v_{2n},\\
&\vdots \\
\nabla_{v_{2n - 1}}\nabla_{v_{2n}}v_{2n - 1} &= \frac{\sqrt{\lambda}}{2}\nabla_{v_{2n - 1}}v_{2n + 1} = \frac{\lambda}{4}v_{2n},\\
\nabla_{v_{2n + 1}}\nabla_{v_{2n}}v_{2n + 1} &= \frac{\sqrt{\lambda}}{2}\nabla_{v_{2n + 1}}v_{2n - 1} = \frac{-\lambda}{4}v_{2n}.
\end{aligned}$$
Therefore
\begin{eqnarray*}
\operatorname{Ric}(v_{2n}) &=& -\left(\displaystyle{\sum_{k = 1}^{n - 1}\frac{-\lambda}{2\sigma_k^2}}\right)v_{2n} - \frac{\lambda}{2}v_{2n}\\
&=& \left(\left(\displaystyle{\sum_{k = 1}^{n - 1}\frac{\lambda}{2\sigma_k^2}}\right) - \frac{\lambda}{2}\right)v_{2n} = \theta v_{2n}.
\end{eqnarray*}
\end{proof}
\begin{remark}
For $n = 1$, we recover the Ricci operator of the Lorentzian inner product $\nu$ on the three-dimensional Heisenberg Lie algebra described (up to permutation) in \cite{ayad2026lorentzian} by
$$\operatorname{Ric} = \operatorname{diag}\left\lbrace \frac{\lambda}{2}, \frac{-\lambda}{2}, \frac{\lambda}{2}\right\rbrace.$$
\end{remark}
\begin{theorem}
The Lorentzian inner product $\nu$ is an algebraic Ricci soliton only for $n = 1$. In this case, $\nu$ is given by $\nu = \operatorname{diag}\left\lbrace 1, -1, \lambda\right\rbrace$. Its algebraic Ricci soliton equation is given by (we recover the result in \cite{ayad2026lorentzian})
$$\operatorname{Ric} = \operatorname{diag}\left\lbrace \frac{\lambda}{2}, \frac{-\lambda}{2}, \frac{\lambda}{2}\right\rbrace = \frac{3\lambda}{2}I_3 + \operatorname{diag}\left\lbrace -\lambda, -2\lambda, -\lambda\right\rbrace,$$
where $D = \operatorname{diag}\left\lbrace -\lambda, -2\lambda, -\lambda\right\rbrace$ is a derivation of $\mathfrak{h}_3$ with respect to the orthonormal basis $\left\lbrace v_1 = f_1, v_2 = \frac{1}{\sqrt{\lambda}}z, v_3 = g_1\right\rbrace$.
\end{theorem}
\begin{proof}
Assume that $\nu$ is an algebraic Ricci soliton, then there exists $\eta \in \mathbb{R}$ and a derivation $D$ of $\mathfrak{h}_{2n + 1}$ such that
$$\operatorname{Ric} = \eta I_{2n + 1} + D.$$
From this, it follows that
\begin{eqnarray*}
D &=& \operatorname{Ric} - \eta I_{2n + 1}\\
&=& \operatorname{diag}\left\lbrace \frac{-\lambda}{2\sigma_1^2} - \eta, \frac{-\lambda}{2\sigma_1^2} - \eta,\ldots, \frac{-\lambda}{2\sigma_{n - 1}^2} - \eta, \frac{-\lambda}{2\sigma_{n - 1}^2} - \eta, \frac{\lambda}{2} - \eta, \theta - \eta, \frac{\lambda}{2} - \eta\right\rbrace.
\end{eqnarray*}
Since $D$ is a derivation of $\mathfrak{h}_{2n + 1}$ that is diagonal, it obviously satisfies the derivation condition for all vanishing brackets. Next, $D$ satisfies the condition
\begin{eqnarray*}
& [D(v_1), v_2] + [v_1, D(v_2)] = D[v_1, v_2] \\ \Leftrightarrow & \left[ \left(\frac{-\lambda}{2\sigma_1^2} - \eta\right)v_1, v_2\right] + \left[ v_1, \left(\frac{-\lambda}{2\sigma_1^2} - \eta\right)v_2\right] = \frac{\sqrt{\lambda}}{\sigma_1}D(v_{2n})\\
\Leftrightarrow& \left(\frac{-\lambda}{\sigma_1^2} - 2\eta\right)[v_1, v_2] = \frac{\sqrt{\lambda}}{\sigma_1}(\theta - \eta)v_{2n} \\
\Leftrightarrow& \left(\frac{-\lambda}{\sigma_1^2} - 2\eta\right)\frac{\sqrt{\lambda}}{\sigma_1}v_{2n} = \frac{\sqrt{\lambda}}{\sigma_1}(\theta - \eta)v_{2n} \\
\Leftrightarrow& \frac{-\lambda}{\sigma_1^2} - 2\eta = \theta - \eta \Leftrightarrow \eta = \frac{-\lambda}{\sigma_1^2} - \theta.
\end{eqnarray*}
Applying the same reasoning to the other brackets, we obtain that
$$\begin{aligned}
\eta &= \frac{-\lambda}{\sigma_i^2} - \theta, \quad i = 1, \ldots, n - 1\\
\eta &= \lambda - \theta.
\end{aligned}$$
If $\sigma_i \neq \sigma_j$, we have $\eta - \eta = 0 = \frac{-\lambda}{\sigma_i^2} + \frac{\lambda}{\sigma_j^2}$, which is a contradiction. Next, we have $\eta - \eta = 0 = \frac{-\lambda}{\sigma_i^2} - \lambda$, which is a contradiction. For $n = 1$, we have $\theta = \frac{-\lambda}{2}$. We do not have the first conditions on $\eta$, so the only condition on $\eta$ is 
$$\eta = \lambda - \theta = \lambda + \frac{\lambda}{2} = \frac{3\lambda}{2}.$$
The algebraic Ricci soliton equation is given by 
$$\operatorname{Ric} = \operatorname{diag}\left\lbrace \frac{\lambda}{2}, \frac{-\lambda}{2}, \frac{\lambda}{2}\right\rbrace = \frac{3\lambda}{2}I_3 + \operatorname{diag}\left\lbrace -\lambda, -2\lambda, -\lambda\right\rbrace,$$
where $D = \operatorname{diag}\left\lbrace -\lambda, -2\lambda, -\lambda\right\rbrace$ is a derivation of $\mathfrak{h}_3$ with respect to the orthonormal basis $\left\lbrace v_1 = f_1, v_2 = \frac{1}{\sqrt{\lambda}}z, v_3 = g_1\right\rbrace$.
\end{proof}
\subsection{The Lorentzian Lie algebra $\left(\mathfrak{h}_{2n + 1}, \phi\right)$}
An orthonormal basis $\mathscr{B} = \left\lbrace v_1, \ldots, v_{2n + 1}\right\rbrace$ of $\left(\mathfrak{h}_{2n + 1}, \phi\right)$ is given by
$$\begin{aligned}
v_1 &= \frac{1}{\sqrt{\sigma_1}}f_1,
&\qquad
v_2 &= \frac{1}{\sqrt{\sigma_1}}g_1, \\[2pt]
v_3 &= \frac{1}{\sqrt{\sigma_2}}f_2,
&\qquad
v_4 &= \frac{1}{\sqrt{\sigma_2}}g_2, \\
&\vdots \\
v_{2n-3} &= \frac{1}{\sqrt{\sigma_{n-1}}}f_{n-1},
&\qquad
v_{2n-2} &= \frac{1}{\sqrt{\sigma_{n-1}}}g_{n-1}, \\[2pt]
v_{2n-1} &= f_n,
&\qquad
v_{2n} &= \frac{1}{\sqrt{2}}(g_n + z), \\[2pt]
v_{2n+1} &= \frac{1}{\sqrt{2}}(g_n - z).
\end{aligned}$$
Note that: $v_{2n} - v_{2n + 1} = \sqrt{2}z \Longrightarrow z = \frac{1}{\sqrt{2}}(v_{2n} - v_{2n + 1})$. The bracket in the basis $\mathscr{B}$ is given by
$$
[v_1, v_2] = \frac{1}{\sqrt{2}\sigma_1}(v_{2n} - v_{2n + 1}), \quad [v_3, v_4] = \frac{1}{\sqrt{2}\sigma_2}(v_{2n} - v_{2n + 1}), \; \ldots,$$
$$[v_{2n - 3}, v_{2n - 2}] = \frac{1}{\sqrt{2}\sigma_{n - 1}}(v_{2n} - v_{2n + 1}),$$
$$[v_{2n - 1}, v_{2n}] = \frac{1}{2}(v_{2n} - v_{2n + 1}), \quad [v_{2n - 1}, v_{2n + 1}] = \frac{1}{2}(v_{2n} - v_{2n + 1}).
$$
The structure constants are
$$\begin{aligned}
\xi_{12(2n)} &= -\xi_{21(2n)} = \frac{1}{\sqrt{2}\sigma_1},\\
\xi_{12(2n + 1)} &= -\xi_{21(2n + 1)} = \frac{1}{\sqrt{2}\sigma_1},\\
\xi_{34(2n)} &= -\xi_{43(2n)} = \frac{1}{\sqrt{2}\sigma_2},\\
\xi_{34(2n + 1)} &= -\xi_{43(2n + 1)} = \frac{1}{\sqrt{2}\sigma_2},\\
&\vdots \\
\xi_{(2n - 3)(2n - 2)(2n)} &= -\xi_{(2n - 2)(2n - 3)(2n)} = \frac{1}{\sqrt{2}\sigma_{n - 1}},\\
\xi_{(2n - 3)(2n - 2)(2n + 1)} &= -\xi_{(2n - 2)(2n - 3)(2n + 1)} = \frac{1}{\sqrt{2}\sigma_{n - 1}},\\
\xi_{(2n - 1)(2n)(2n)} &= -\xi_{(2n)(2n - 1)(2n)} = \frac{1}{2},\\
\xi_{(2n - 1)(2n)(2n + 1)} &= -\xi_{(2n)(2n - 1)(2n + 1)} = \frac{1}{2},\\
\xi_{(2n - 1)(2n + 1)(2n)} &= -\xi_{(2n + 1)(2n - 1)(2n)} = \frac{1}{2},\\
\xi_{(2n - 1)(2n + 1)(2n + 1)} &= -\xi_{(2n + 1)(2n - 1)(2n + 1)} = \frac{1}{2}.
\end{aligned}$$
\begin{theorem}\label{civita}
The nonzero terms in the Levi-Civita connection of $\left(\mathfrak{h}_{2n + 1}, \phi\right)$ are
$$
\begin{aligned}
\nabla_{v_i}v_i &= 0, \qquad i=1,\ldots,2n - 1,\\
\nabla_{v_{2n}}v_{2n} &= \nabla_{v_{2n + 1}}v_{2n + 1} = \frac{1}{2}v_{2n - 1},\\[2mm]
\nabla_{v_1}v_2
&=-\nabla_{v_2}v_1=\frac{1}{2\sqrt{2}\sigma_1}v_{2n} - \frac{1}{2\sqrt{2}\sigma_1}v_{2n + 1},\\
\nabla_{v_1}v_{2n}
&=\nabla_{v_{2n}}v_1=\frac{-1}{2\sqrt{2}\sigma_1}v_2,\\
\nabla_{v_2}v_{2n}
&=\nabla_{v_{2n}}v_2=\frac{1}{2\sqrt{2}\sigma_1}v_1,\\
\nabla_{v_1}v_{2n + 1}
&=\nabla_{v_{2n + 1}}v_1=\frac{-1}{2\sqrt{2}\sigma_1}v_2,\\
\nabla_{v_2}v_{2n + 1}
&=\nabla_{v_{2n + 1}}v_2=\frac{1}{2\sqrt{2}\sigma_1}v_1,\\[2mm]
\nabla_{v_3}v_4
&=-\nabla_{v_4}v_3=\frac{1}{2\sqrt{2}\sigma_2}v_{2n} - \frac{1}{2\sqrt{2}\sigma_2}v_{2n + 1},\\
\nabla_{v_3}v_{2n}
&=\nabla_{v_{2n}}v_3=\frac{-1}{2\sqrt{2}\sigma_2}v_4,\\
\nabla_{v_4}v_{2n}
&=\nabla_{v_{2n}}v_4=\frac{1}{2\sqrt{2}\sigma_2}v_3,\\
\nabla_{v_3}v_{2n + 1}
&=\nabla_{v_{2n + 1}}v_3=\frac{-1}{2\sqrt{2}\sigma_2}v_4,\\
\nabla_{v_4}v_{2n + 1}
&=\nabla_{v_{2n + 1}}v_4=\frac{1}{2\sqrt{2}\sigma_2}v_3,\\
&\hspace{2cm}\vdots\\[-1mm]
\nabla_{v_{2n - 3}}v_{2n - 2}
&=-\nabla_{v_{2n - 2}}v_{2n - 3}=\frac{1}{2\sqrt{2}\sigma_{n - 1}}v_{2n} - \frac{1}{2\sqrt{2}\sigma_{n - 1}}v_{2n + 1},\\
\nabla_{v_{2n - 3}}v_{2n}
&=\nabla_{v_{2n}}v_{2n - 3}=\frac{-1}{2\sqrt{2}\sigma_{n - 1}}v_{2n - 2},\\
\nabla_{v_{2n - 2}}v_{2n}
&=\nabla_{v_{2n}}v_{2n - 2}=\frac{1}{2\sqrt{2}\sigma_{n - 1}}v_{2n - 3},\\
\nabla_{v_{2n - 3}}v_{2n + 1}
&=\nabla_{v_{2n + 1}}v_{2n - 3}=\frac{-1}{2\sqrt{2}\sigma_{n - 1}}v_{2n - 2},\\
\nabla_{v_{2n - 2}}v_{2n + 1}
&=\nabla_{v_{2n + 1}}v_{2n - 2}=\frac{1}{2\sqrt{2}\sigma_{n - 1}}v_{2n - 3},\\[2mm]
\nabla_{v_{2n-1}}v_{2n}
&=\nabla_{v_{2n - 1}}v_{2n+1}=0,\\
\nabla_{v_{2n}}v_{2n-1}
&=\nabla_{v_{2n+1}}v_{2n-1}=\frac{-1}{2}v_{2n} + \frac{1}{2}v_{2n + 1},\\
\nabla_{v_{2n}}v_{2n+1}
&=\nabla_{v_{2n+1}}v_{2n}=\frac{1}{2}v_{2n-1}.
\end{aligned}
$$
\end{theorem}
\begin{proof}
Using the structure constants described above, we obtain the result by applying formula (\ref{Levi}).
\end{proof}
\subsubsection{The Ricci operator of $\left(\mathfrak{h}_{2n + 1}, \phi\right)$}
\begin{theorem}
The Ricci operator of $\left(\mathfrak{h}_{2n + 1}, \phi\right)$ is represented in the basis $\mathscr{B}$ by the following matrix
$$\operatorname{Ric} = \begin{bmatrix}
 & 0 &  &  &  &  &  \\
 & & 0 &  &  &  & \\
 & &  & \ddots &  & \\
 & &  &  & 0 &  &  & \\
 & &  &  &   & \alpha & \alpha\\
 & &  &  &   & -\alpha & -\alpha
\end{bmatrix},$$
where $\alpha$ is given by: $\alpha = \displaystyle{\sum_{k = 1}^{n - 1}\frac{1}{4\sigma_k^2}}$.
\end{theorem}
\begin{proof}
Let us compute $\operatorname{Ric}(v_1)$, we have 
$$\operatorname{Ric}(v_1) = \displaystyle{\sum_{k = 2}^{2n + 1} \phi(v_k, v_k)R_{v_kv_1}v_k},$$
where
$$R_{v_kv_1}v_k = \nabla_{[v_k, v_1]}v_k - \nabla_{v_k}\nabla_{v_1}v_k + \nabla_{v_1}\nabla_{v_k}v_k.$$
According to the bracket in $\mathscr{B}$ and the Levi-Civita connection, we see that the only nonzero terms in $R_{v_kv_1}v_k$ are the following
$$\begin{aligned}
R_{v_2v_1}v_2 &= \nabla_{[v_2, v_1]}v_2 - \nabla_{v_2}\nabla_{v_1}v_2\\
&= \frac{-1}{\sqrt{2}\sigma_1}\nabla_{v_{2n}}v_2 + \frac{1}{\sqrt{2}\sigma_1}\nabla_{v_{2n + 1}}v_2 - \nabla_{v_2}\left(\frac{1}{2\sqrt{2}\sigma_1}v_{2n} - \frac{1}{2\sqrt{2}\sigma_1}v_{2n + 1}\right)\\
&= \frac{-1}{\sqrt{2}\sigma_1}\nabla_{v_{2n}}v_2 + \frac{1}{\sqrt{2}\sigma_1}\nabla_{v_{2n + 1}}v_2 - \frac{1}{2\sqrt{2}\sigma_1}\nabla_{v_{2}}v_{2n} + \frac{1}{2\sqrt{2}\sigma_1}\nabla_{v_{2}}v_{2n + 1}\\
&= \frac{1}{\sqrt{2}\sigma_1}\left(\nabla_{v_{2n + 1}}v_2 - \nabla_{v_{2n}}v_2\right) + \frac{1}{2\sqrt{2}\sigma_1}\left(\nabla_{v_{2}}v_{2n + 1} - \nabla_{v_{2}}v_{2n}\right)\\
&= 0.\\
R_{v_{2n}v_1}v_{2n} &= - \nabla_{v_{2n}}\nabla_{v_1}v_{2n} + \nabla_{v_{1}}\nabla_{v_{2n}}v_{2n}\\
&= \frac{1}{2\sqrt{2}\sigma_1}\nabla_{v_{2n}}v_{2} + \frac{1}{2}\underbrace{\nabla_{v_1}v_{2n - 1}}_{= 0}\\
&= \frac{1}{8\sigma_1^2}v_1.\\
R_{v_{2n + 1}v_1}v_{2n + 1} &= - \nabla_{v_{2n + 1}}\nabla_{v_1}v_{2n + 1} + \nabla_{v_{1}}\nabla_{v_{2n + 1}}v_{2n + 1}\\
&= \frac{1}{2\sqrt{2}\sigma_1}\nabla_{v_{2n + 1}}v_{2} + \frac{1}{2}\underbrace{\nabla_{v_1}v_{2n - 1}}_{= 0}\\
&= \frac{1}{8\sigma_1^2}v_1.
\end{aligned}$$
Since $v_{2n + 1}$ is timelike, then 
$$\operatorname{Ric}(v_1) = R_{v_{2n}v_1}v_{2n} - R_{v_{2n + 1}v_1}v_{2n + 1} = 0.$$
Using the same reasoning, we find that
$$\operatorname{Ric}(v_i) = 0,\; \forall i = 1, \ldots, 2n - 1.$$
Next, let us compute $\operatorname{Ric}(v_{2n})$, we have
$$\operatorname{Ric}(v_{2n}) = \displaystyle{\sum_{k = 1}^{2n + 1} \phi(v_k, v_k)R_{v_kv_{2n}}v_k}.$$
After some careful calculations, we find that
$$\begin{aligned}
R_{v_1v_{2n}}v_1 &= \frac{1}{8\sigma_1^2}v_{2n} - \frac{1}{8\sigma_1^2}v_{2n + 1},\\
R_{v_2v_{2n}}v_2 &= \frac{1}{8\sigma_1^2}v_{2n} - \frac{1}{8\sigma_1^2}v_{2n + 1},\\
&\vdots \\
R_{v_{2n - 3}v_{2n}}v_{2n - 3} &= \frac{1}{8\sigma_{n - 1}^2}v_{2n} - \frac{1}{8\sigma_{n - 1}^2}v_{2n + 1},\\
R_{v_{2n - 2}v_{2n}}v_{2n - 2} &= \frac{1}{8\sigma_{n - 1}^2}v_{2n} - \frac{1}{8\sigma_{n - 1}^2}v_{2n + 1},\\
R_{v_{2n - 1}v_{2n}}v_{2n - 1} &= R_{v_{2n + 1}v_{2n}}v_{2n + 1} = 0.
\end{aligned}$$
Therefore
$$\begin{aligned}
\operatorname{Ric}(v_{2n}) &= \displaystyle{\sum_{k = 1}^{2n - 2} R_{v_kv_{2n}}v_k}\\
&= \displaystyle{\sum_{k = 1}^{n - 1} \left(\frac{1}{4\sigma_{k}^2}v_{2n} - \frac{1}{4\sigma_{k}^2}v_{2n + 1}\right)}\\
&= \displaystyle{\sum_{k = 1}^{n - 1} \frac{1}{4\sigma_{k}^2}}v_{2n} - \displaystyle{\sum_{k = 1}^{n - 1} \frac{1}{4\sigma_{k}^2}}v_{2n + 1}\\
&= \alpha v_{2n} - \alpha v_{2n + 1}.
\end{aligned}$$
Finally, for $\operatorname{Ric}(v_{2n + 1})$, we obtain that
$$\begin{aligned}
R_{v_1v_{2n + 1}}v_1 &= \frac{1}{8\sigma_1^2}v_{2n} - \frac{1}{8\sigma_1^2}v_{2n + 1},\\
R_{v_2v_{2n + 1}}v_2 &= \frac{1}{8\sigma_1^2}v_{2n} - \frac{1}{8\sigma_1^2}v_{2n + 1},\\
&\vdots \\
R_{v_{2n - 3}v_{2n + 1}}v_{2n - 3} &= \frac{1}{8\sigma_{n - 1}^2}v_{2n} - \frac{1}{8\sigma_{n - 1}^2}v_{2n + 1},\\
R_{v_{2n - 2}v_{2n + 1}}v_{2n - 2} &= \frac{1}{8\sigma_{n - 1}^2}v_{2n} - \frac{1}{8\sigma_{n - 1}^2}v_{2n + 1},\\
R_{v_{2n - 1}v_{2n + 1}}v_{2n - 1} &= R_{v_{2n}v_{2n + 1}}v_{2n} = 0.
\end{aligned}$$
Therefore
$$\begin{aligned}
\operatorname{Ric}(v_{2n + 1}) &= \displaystyle{\sum_{k = 1}^{2n - 2} R_{v_kv_{2n + 1}}v_k}\\
&= \displaystyle{\sum_{k = 1}^{n - 1} \left(\frac{1}{4\sigma_{k}^2}v_{2n} - \frac{1}{4\sigma_{k}^2}v_{2n + 1}\right)}\\
&= \alpha v_{2n} - \alpha v_{2n + 1}.
\end{aligned}$$
Consequently, the Ricci operator $\operatorname{Ric}$ of $\left(\mathfrak{h}_{2n + 1}, \phi\right)$ is represented in the basis $\mathscr{B}$ by the matrix given in the theorem above.
\end{proof}
\begin{theorem}
The Lorentzian inner product $\phi$ is an algebraic Ricci soliton where $\eta = 0$ and $D = \operatorname{Ric}$ is a derivation of $\mathfrak{h}_{2n + 1}$ with respcet to the orthonormal basis $\mathscr{B}$. Moreover, $\phi$ is a steady algebraic Ricci soliton.
\end{theorem}
\begin{proof}
Assume that $\operatorname{Ric} = \eta I_{2n + 1} + D$, where $\eta \in \mathbb{R}$ and $D$ is a derivation of $\mathfrak{h}_{2n + 1}$. Then 
$$D = \operatorname{Ric} - \eta I_{2n + 1} = \begin{bmatrix}
& -\eta &  &  &  &  &  \\
& & -\eta &  &  &  & \\
& &  & \ddots &  & \\
& &  &  & -\eta &  &  & \\
& &  &  &   & \alpha - \eta & \alpha\\
& &  &  &   & -\alpha & -\alpha - \eta
\end{bmatrix}.$$
Given that $D$ is a derivation of $\mathfrak{h}_{2n + 1}$, the condition below holds
\begin{eqnarray*}
& [D(v_1), v_2] + [v_1, D(v_2)] = D[v_1, v_2] \\ \Leftrightarrow & \left[-\eta v_1, v_2\right] + \left[ v_1, -\eta v_2\right] = \frac{1}{\sqrt{2}\sigma_1}\left(D(v_{2n}) - D(v_{2n + 1})\right)\\
\Leftrightarrow& -2\eta[v_1, v_2] = \frac{1}{\sqrt{2}\sigma_1}\left((\alpha - \eta)v_{2n} - \alpha v_{2n + 1} - \alpha v_{2n} + (\alpha + \eta)v_{2n + 1}\right)\\
\Leftrightarrow& -2\eta\frac{1}{\sqrt{2}\sigma_1}\left(v_{2n} - v_{2n + 1}\right) = \frac{1}{\sqrt{2}\sigma_1}\left(-\eta v_{2n} + \eta v_{2n + 1}\right)\\
\Leftrightarrow& -2\eta\frac{1}{\sqrt{2}\sigma_1}\left(v_{2n} - v_{2n + 1}\right) = \frac{-\eta}{\sqrt{2}\sigma_1}\left(v_{2n} - v_{2n + 1}\right)\\
\Leftrightarrow& 2\eta = \eta \Leftrightarrow \eta = 0.
\end{eqnarray*}
The conditions 
$$[D(v_i), v_{i + 1}] + [v_i, D(v_{i + 1})] = D[v_i, v_{i + 1}], \quad \forall i = 3, 5, \ldots, 2n - 3$$
are similar to the previous condition. Hence $\eta = 0$ and $D = \operatorname{Ric}$. Next, we have
\begin{eqnarray*}
& [\underbrace{D(v_{2n - 1})}_{= 0}, v_{2n}] + [v_{2n - 1}, D(v_{2n})] = D[v_{2n - 1}, v_{2n}] \\ \Leftrightarrow & \left[ v_{2n - 1}, \alpha v_{2n} - \alpha v_{2n + 1}\right] = \frac{1}{2}\left(D(v_{2n}) - D(v_{2n + 1})\right)\\
\Leftrightarrow& \alpha[v_{2n - 1}, v_{2n}] - \alpha[v_{2n - 1}, v_{2n + 1}]  = \frac{1}{2}\left(\alpha v_{2n} - \alpha v_{2n + 1} - \alpha v_{2n} + \alpha v_{2n + 1}\right)\\
\Leftrightarrow& 0 = 0.
\end{eqnarray*}
Hence, this condition is satisfied by $D = \operatorname{Ric}$. Similarly, we see that the following condition is satisfied
$$[D(v_{2n - 1}), v_{2n + 1}] + [v_{2n - 1}, D(v_{2n + 1})] = D[v_{2n - 1}, v_{2n + 1}].$$
It is straightforward to check thad $D = \operatorname{Ric}$ satisfies the derivation condition for all vanishing brackets. For example, we have
$$\begin{aligned}
[\operatorname{Ric}(v_1), v_3] + [v_1, \operatorname{Ric}(v_3)] &= 0, \quad\text{satisfied}\\
&\vdots\\
[\operatorname{Ric}(v_1), v_{2n + 1}] + [v_1, \operatorname{Ric}(v_{2n + 1})] &= 0, \quad\text{satisfied}\\
&\vdots\\
[\operatorname{Ric}(v_{2n}), v_{2n + 1}] + [v_{2n}, \operatorname{Ric}(v_{2n + 1})] &= 0, \quad\text{satisfied}.
\end{aligned}$$
Consequently, the Lorentzian inner product $\phi$ is a steady algebraic Ricci soliton.
\end{proof}
\begin{remark}
For $n = 1$, $\phi$ is represented in the basis $\left\lbrace f_1, g_1, z\right\rbrace$ by the matrix
$$\phi = \begin{bmatrix}
1 & 0 & 0\\
0 & 0 & 1\\
0 & 1 & 0
\end{bmatrix}.$$
An orthonormal basis of $(\mathfrak{h}_3, \phi)$ is given by
$$v_1 = f_1, \qquad v_2 = \frac{1}{\sqrt{2}}(g_1 + z), \qquad v_3 = \frac{1}{\sqrt{2}}(g_1 - z).$$
The bracket is described by
$$[v_1, v_2] = [v_1, v_3] = \frac{1}{2}(v_2 - v_3).$$
By theorem \ref{civita}, we recover the following Levi-Civita connection of $(\mathfrak{h}_3, \phi)$
\begin{multicols}{2}
$\nabla_{v_1}v_1 = 0$\par
$\nabla_{v_2}v_2 = \nabla_{v_3}v_3 = \frac{1}{2}v_1$\par
$\nabla_{v_1}v_2 = \nabla_{v_1}v_3 = 0$\par
$\nabla_{v_2}v_1 = \nabla_{v_3}v_1 = \frac{-1}{2}v_2 + \frac{1}{2}v_3$\par
$\nabla_{v_2}v_3 = \nabla_{v_3}v_2 = \frac{1}{2}v_1$
\end{multicols}
Now, it is straightforward to check that
$$R_{v_1v_2} = R_{v_1v_3} = R_{v_2v_3} = 0.$$
Consequently, $(\mathfrak{h}_3, \phi)$ is flat.
\end{remark}

\end{document}